\documentclass{louloupart1}
\usepackage{newtxtext,newtxmath}
\usepackage{pinlabel}
\usepackage{graphicx}
\usepackage[all]{xy}

\title[Parabolic subgroups and word problem in virtual Artin groups]{Parabolic subgroups and word problem in virtual Artin groups}
\author[J. G\'alvez Mateos]{José G\'alvez Mateos}
\givenname{Jos\'e}
\surname{G\'alvez Mateos}
\address{Jos\'e G\'alvez Mateos, Instituto de Matem\'aticas de la Universidad de Sevilla (IMUS) and Departamento de \'Algebra, Universidad de Sevilla, Spain.}
\email{jgalvez7@us.es}
\author[F. Gavazzi]{Federica Gavazzi}
\givenname{Federica}
\surname{Gavazzi}
\address{Federica Gavazzi, Department of Mathematics, Heriot-Watt University and Maxwell Institute for Mathematical Sciences, Edinburgh, UK.}
\email{F.Gavazzi@hw.ac.uk}
\author[L. Paris]{Luis Paris}
\givenname{Luis}
\surname{Paris}
\address{Luis Paris, Universit\'e Bourgogne Europe, CNRS, IMB, UMR 5584, 21000 Dijon, France.}
\email{lparis@u-bourgogne.fr}

\newtheorem{thm}{Theorem}[section]
\newtheorem{lem}[thm]{Lemma}
\newtheorem{prop}[thm]{Proposition}
\newtheorem{corl}[thm]{Corollary}

\theoremstyle{definition}
\newtheorem{rem}[thm]{Remark}
\newtheorem{expl}[thm]{Example}

\theoremstyle{definition2}
\newtheorem*{acknow}{Acknowledgments}
\newtheorem*{defin}{Definition}

\numberwithin{equation}{section}

\makeatletter
\renewcommand{\thefigure}{\ifnum \c@section>\z@ \thesection.\fi
 \@arabic\c@figure}
\@addtoreset{figure}{section}
\makeatother

\begin{document}

\def\N{\mathbb N} \def\VA{{\rm VA}} \def\SS{\mathcal S}
\def\TT{\mathcal T} \def\KVA{{\rm KVA}} \def\PVA{{\rm PVA}}
\def\id{{\rm id}} \def\R{\mathbb R} \def\GL{{\rm GL}}
\def\RR{\mathcal R} \def\supp{{\rm supp}} \def\dpt{{\rm dpt}}
\def\XX{\mathcal X} \def\K{\mathbb K} \def\AA{\mathcal A}
\def\YY{\mathcal Y}


\begin{abstract}
We begin by establishing two fundamental results on standard parabolic subgroups of virtual Artin groups. 
We first show that a standard parabolic subgroup is naturally isomorphic to a virtual Artin group.
Second, we prove that the intersection of two standard parabolic subgroups is a standard parabolic subgroup.
Our main result is that, if all free of infinity standard parabolic subgroups of a given virtual Artin group $\VA [\Gamma]$ have a solvable word problem, then $\VA [\Gamma]$ itself has a solvable word problem.
It follows that virtual Artin groups of FC type and, more generally, of affine-FC type, have a solvable word problem.
We also prove that, if a virtual Artin group $\VA [\Gamma]$ has a solvable word problem, then the strong membership problem for any standard parabolic subgroup in $\VA [\Gamma]$ is solvable.



\end{abstract}

\maketitle


\section{Introduction}\label{sec1}

Let $S$ be a finite set.
A \emph{Coxeter matrix} over $S$ is a symmetric square matrix $M = (m_{s,t})_{s,t \in S}$ indexed by the elements of $S$, with coefficients in $\N_{\ge 1} \cup \{ \infty\}$, such that $m_{s,s}=1$ for all $s \in S$ and $m_{s,t} = m_{t,s} \ge 2$ for all $s, t \in S$ such that $s \neq t$.
Such a matrix is typically represented by a labeled graph $\Gamma$, called a \emph{Coxeter graph}, which is defined as follows.
The set of vertices of $\Gamma$ is $S$, two distinct vertices $s,t \in S$ are connected by an edge if $m_{s,t} \ge 3$, and an edge $\{s,t\}$ is labeled with $m_{s,t}$ if $m_{s,t} \ge 4$.
From now on we fix a Coxeter graph $\Gamma$ and we denote by $M=(m_{s,t})_{s,t \in S}$ its associated Coxeter matrix.

For two letters $a$ and $b$ and an integer $m \ge 2$, we denote by $\Pi(a,b,m)$ the alternating product $aba \cdots$ of length $m$.
In other words, $\Pi(a, b, m) = (ab)^{\frac{m}{2}}$ if $m$ is even, and $\Pi(a,b,m) = (ab)^{\frac{m-1}{2}} a$ if $m$ is odd.
We associate with the Coxeter graph $\Gamma$ an \emph{Artin group}, denoted by $A[\Gamma]$, and a \emph{Coxeter group}, denoted by $W[\Gamma]$.
The Artin group is defined by the following presentation:
\[ 
A [\Gamma] = \langle S \mid \Pi (s, t, m_{s,t}) = \Pi (t, s, m_{s,t}) \text{ for } s,t \in S \text{ such that } s \neq t \text{ and } m_{s,t} \neq \infty \rangle \,.
\]
The Coxeter group $W[\Gamma]$ is the quotient of $A [\Gamma]$ by the relations $s^2=1$, for all $s \in S$.

Coxeter groups were introduced by Tits \cite{Tit13} in a preprint written in the late 1950s but only published in 2013.
This preprint served as the basis for Bourbaki's seminal book on the theory of Coxeter groups \cite{Bou68}.
Despite a good understanding of these groups, they remain an active area of research. 
On the other hand, Artin groups were introduced by Tits in \cite{Tit66} as extensions of Coxeter groups, but their study really began in the 1970s with the work of Brieskorn \cite{Bri71, Bri73}, Brieskorn--Saito \cite{BS72}, and Deligne \cite{Del72}. 
Unlike Coxeter groups, for which many fundamental results are known, Artin groups remain partially understood and they are currently the subject of a flourishing field of study, particularly in geometric group theory.

In \cite{BPT23}, the authors associate with the Coxeter graph $\Gamma$ another group, called the \emph{virtual Artin group} of $\Gamma$ and denoted by $\VA[\Gamma]$. 
Following the model of virtual braid groups introduced by Kauffman in \cite{Kau99}, the underlying idea of its definition is to relate a copy of the Artin group $A[\Gamma]$ to a copy of the Coxeter group $W[\Gamma]$ using mixed relations that mimic the action of $W[\Gamma]$ on its root system.
Its formal definition is as follows.

\begin{defin}
Define the following sets in one-to-one correspondence with $S$:
\[
\SS = \{ \sigma_s \mid s \in S\}\,, \quad \TT = \{ \tau_s \mid s \in S \}\,.
\]
The \emph{virtual Artin group} of $\Gamma$, denoted by $\VA [\Gamma]$, is the group defined by the presentation with generating set $\SS \sqcup \TT$ and the following relations:
\begin{itemize}
\item[(v1)]
$\Pi (\sigma_s, \sigma_t, m_{s,t}) = \Pi (\sigma_t, \sigma_s, m_{s,t})$ for all $s,t \in S$ such that $s \neq t$ and $m_{s,t} \neq \infty$;
\item[(v2)]
$\tau_s^2 = 1$ for all $s \in S$; $\Pi(\tau_s, \tau_t, m_{s,t}) = \Pi(\tau_t, \tau_s, m_{s,t})$ for all $s,t \in S$ such that $s \neq t$ and $m_{s,t} \neq \infty$;
\item[(v3)]
$\sigma_s \, \Pi(\tau_t, \tau_s, m_{s,t}-1) = \Pi(\tau_t, \tau_s, m_{s,t}-1) \, \sigma_r$, where $r=s$ if $m_{s,t}$ is even and $r=t$ if $m_{s,t}$ is odd, for all $s,t \in S$ such that $s \neq t$ and $m_{s,t} \neq \infty$.
\end{itemize}
\end{defin}

Let $X \subseteq S$.
We denote by $\Gamma_X$ the full Coxeter subgraph of $\Gamma$ spanned by $X$, by $W_X [\Gamma]$ the subgroup of $W [\Gamma]$ generated by $X$, and by $A_X [\Gamma]$ the subgroup of $A [\Gamma]$ generated by $X$.
The subgroup $W_X [\Gamma]$ (resp. $A_X [\Gamma]$) is called a \emph{standard parabolic subgroup} of $W [\Gamma]$ (resp. $A [\Gamma]$).
These subgroups play a crucial role in the theory of both Coxeter groups and Artin groups.
It is known from \cite{Bou68} that the natural homomorphism $W[\Gamma_X] \to W_X[\Gamma]$ which maps $x$ to $x$ for all $x \in X$ is an isomorphism.
Similarly, it is known from \cite{vLek83} that the natural homomorphism $A[\Gamma_X] \to A_X[\Gamma]$ which maps $x$ to $x$ for all $x \in X$ is an isomorphism.
The first step in our study is to establish a similar result for virtual Artin groups.

\begin{defin}
Let $X \subseteq S$.
Set $\SS_X = \{ \sigma_x \mid x \in X\}$ and $\TT_X = \{ \tau_x \mid x \in X\}$.
The subgroup of $\VA [\Gamma]$ generated by $\SS_X \sqcup \TT_X$ is denoted by $\VA_X [\Gamma]$ and is called a \emph{standard parabolic subgroup} of $\VA [\Gamma]$.
\end{defin}

\begin{thm}\label{thm1_1}
Let $X \subseteq S$.
The natural homomorphism $\VA [\Gamma_X] \to \VA_X [\Gamma]$ that maps $\sigma_x$ to $\sigma_x$ and $\tau_x$ to $\tau_x$ for all $x \in X$ is an isomorphism.
\end{thm}

Another classical result for Coxeter groups (see \cite{Bou68}) and Artin groups (see \cite{vLek83}) states that, for any $X, Y \subseteq S$, we have $W_X [\Gamma] \cap W_Y [\Gamma] = W_{X \cap Y} [\Gamma]$ and $A_X [\Gamma] \cap A_Y [\Gamma] = A_{X \cap Y} [\Gamma]$.
We also establish a similar result for virtual Artin groups.

\begin{thm}\label{thm1_2}
Let $X,Y \subseteq S$.
Then $\VA_X [\Gamma] \cap \VA_Y [\Gamma] = \VA_{X \cap Y} [\Gamma]$.
\end{thm}

\begin{rem}\label{rem1_3}
A \emph{parabolic subgroup} of $W[\Gamma]$ is defined to be a conjugate of a standard parabolic subgroup of $W [\Gamma]$.
Similarly, a \emph{parabolic subgroup} of $A[\Gamma]$ is defined to be a conjugate of a standard parabolic subgroup of $A[\Gamma]$, and a \emph{parabolic subgroup} of $\VA[\Gamma]$ is defined to be a conjugate of a standard parabolic subgroup of $\VA[\Gamma]$.
It is known from \cite{Sol76} that the intersection of two parabolic subgroups of $W[\Gamma]$ is itself a parabolic subgroup of $W[\Gamma]$.
It is conjectured that the same result holds for Artin groups.
This conjecture has been proven in numerous cases (see \cite{CGGW19, Mor21, AF22, Blu22, CMV23}) and it is a hot topic in this area of research.
However, such a result does not hold for virtual Artin groups.
A relatively simple counterexample is presented at the end of Section \ref{sec3} (see Example \ref{expl3_9}).
\end{rem}

A subset $X \subseteq S$ is called \emph{free of infinity} if, for all $s,t \in X$, we have $m_{s,t} \neq \infty$.
Artin groups often satisfy the principle that, if a property holds for all subgroups $A_X [\Gamma] \simeq A [\Gamma_X]$ with $X$ free of infinity, then it holds for $A [\Gamma]$ itself (see \cite{GP12}).
The same principle is expected to apply to virtual Artin groups.
The main result of this paper is that this principle holds for the solvability of the word problem. 

\begin{thm}\label{thm1_4}
If the word problem in $\VA_X [\Gamma] \simeq \VA [\Gamma_X]$ is solvable for every subset $X \subseteq S$ that is free of infinity, then the word problem in $\VA [\Gamma]$ is solvable.
\end{thm}

A Coxeter graph $\Omega$ is said to be of \emph{spherical type} (resp. \emph{affine type}) if $W[\Omega]$ is finite (resp. affine).
We say that $\Gamma$ is of \emph{ FC type} if $\Gamma_X$ is of spherical type for every subset  $X \subseteq S$ that is free of infinity.
Similarly, we say that $\Gamma$ is of \emph{affine-FC type} if $\Gamma_X$ is either of spherical or affine type for every $X \subseteq S$ that is free of infinity.
Note that any Coxeter graph of FC type is of affine-FC type.
Artin groups of FC type were introduced by Charney and Davis \cite{CD95} in their study of the $K(\pi, 1)$ problem for Artin groups.
They act on CAT(0) cube complexes, and these actions are an essential ingredient in their study.
The notion of affine-FC type is new, but it arises naturally in our context; indeed:

\begin{corl}\label{corl1_5}
If $\Gamma$ is of affine-FC type, then $\VA [\Gamma]$ has a solvable word problem.
\end{corl}

\begin{proof}
Suppose that $\Gamma$ is of affine-FC type.
It is known from \cite[Theorem 5.1]{BPT23} that, if $\Omega$ is a Coxeter graph of spherical or affine type, then $\VA [\Omega]$ has a solvable word problem.
This implies that $\VA_X [\Gamma] \simeq \VA [\Gamma_X]$ has a solvable word problem for every $X \subseteq S$ that is free of infinity, hence, by Theorem \ref{thm1_4}, $\VA [\Gamma]$ has a solvable word problem.
\end{proof}

If $G$ is a group, $S_G$ a generating set for $G$, and $\omega$ a word over $S_G \sqcup S_G^{-1}$, then we denote by $\overline{\omega}$ the element of $G$ represented by $\omega$.
Let $G$ be a group and $H$ a subgroup of $G$, both finitely generated.
Let $S_G$ be a finite generating set for $G$ and $S_H$ a finite generating set for $H$.
We say that the \emph{membership problem} for $H$ in $G$ is solvable if there exists an algorithm that, given a word $\omega$ over $S_G \sqcup S_G^{-1}$, decides whether $\overline{\omega}$ belongs to $H$ or not. 
Furthermore, we say that the \emph{strong membership problem} for $H$ in $G$ is solvable if there exists an algorithm that, given a word $\omega$ over $S_G \sqcup S_G^{-1}$, decides whether $\overline{\omega}$ belongs to $H$ and, if so, determines a word $\mu$ over $S_H \sqcup S_H^{-1}$ such that $\overline{\mu} = \overline{\omega}$.

The following result is a preliminary to the proof of Theorem \ref{thm1_4}, but it is also of independent interest. 

\begin{thm}\label{thm1_6}
Let $X \subseteq S$.
If the word problem in $\VA [\Gamma]$ is solvable, then the strong membership problem for $\VA_X [\Gamma]$ in $\VA [\Gamma]$ is solvable.
\end{thm}

The paper is organized as follows.
Section \ref{sec2} is divided into four subsections.
In the first subsection, we show that, for $X \subseteq S$, the strong membership problem for $W_X [\Gamma]$ in $W [\Gamma]$ is solvable (see Proposition \ref{prop2_5}).
This result is well-known to experts, but to the best of our knowledge, it is nowhere explicitly stated and proven in the literature.
In the second subsection, we show that, for $X \subseteq S$, the strong membership problem for $A_X [\Gamma]$ in $A [\Gamma]$ is solvable provided that the word problem in $A [\Gamma]$ is solvable (see Proposition \ref{prop2_13}).
This result is new, but it is implicit in \cite{CP14}, \cite{BP23}, or \cite{God23}.
Subsection \ref{subsec2_3} provides preliminaries on virtual Artin groups from \cite{BPT23}.
Subsection \ref{subsec2_4} explains how to solve the word problem in an amalgamated product under certain conditions (see Proposition \ref{prop2_18}).
This result is important for proving Theorem \ref{thm1_4} from Theorem \ref{thm1_6}.
We prove Theorems \ref{thm1_1} and \ref{thm1_2} in Section \ref{sec3}, and we prove Theorems \ref{thm1_6} and \ref{thm1_4} in Section \ref{sec4}.

\begin{acknow}
The authors would like to thank Mar\'ia Cumplido for several stimulating and constructive exchanges.
The first author is part of the research project PID2022-138719NA-I00, financed by MCIN/AEI/10.13039/501100011033/FEDER, UE.
The second author was supported by the EPSRC Standard Research Grant UKRI1018.
The third author is partially supported by the French project ``CaGeT'' (ANR-25-CE40-4162) of the ANR.
The IMB receives support from the EIPHI Graduate School (contract ANR-17-EURE-0002).
\end{acknow}


\section{Preliminaries}\label{sec2}

In our study we will use the fact that, for $X \subseteq S$, the strong membership problem for $W_X [\Gamma]$ in $W [\Gamma]$ is solvable.
This result is known to experts but not explicitly stated, hence we provide a proof in Subsection \ref{subsec2_1}.
Furthermore, the tools used in this proof also prove that $W [\Gamma_X]$ embeds into $W [\Gamma]$ and that the intersection of two standard parabolic subgroups of $W [\Gamma]$ is itself a standard parabolic subgroup. 
Although these two results are well-known, their proofs are short and natural in our context, hence we include them for the sake of completeness.

The goal of Subsection \ref{subsec2_2} is to prove that, for $X \subseteq S$, the strong membership problem for $A_X [\Gamma]$ in $A [\Gamma]$ is solvable, provided that the word problem in $A [\Gamma]$ is solvable (see Proposition \ref{prop2_13}).
This result is new, although it is implicit in \cite{BP23} and \cite{CP14}.
Furthermore, as in the case of Coxeter groups, the techniques used in the proof also prove that $A [\Gamma_X]$ embeds into $A [\Gamma]$ and that the intersection of two standard parabolic subgroups of $A [\Gamma]$ is itself a standard parabolic subgroup. 
We also provide short proofs of these two results, as they may be of independent interest to the reader.

In Subsection \ref{subsec2_3}, we present some results from \cite{BPT23} that enable us to reduce the study of virtual Artin groups to that of Coxeter and Artin groups.
The goal of Subsection \ref{subsec2_4} is to prove Proposition \ref{prop2_18}, which concerns amalgamated products.
While this result is known to experts, we have not found it stated in this form in the literature.
It is important in our study, hence we provide a detailed proof.

\subsection{Standard parabolic subgroups of Coxeter groups}\label{subsec2_1}

The length of a word $\omega$ will always be denoted by $\ell (\omega)$.
We denote by $\ell_S : W [\Gamma] \to \N$ the word length relative to $S$.
As usual, if $\omega$ is a word over $S$, then we denote by $\overline{\omega}$ the element of $W [\Gamma]$ represented by $\omega$.
Note that, as $S^{-1} = S$, we only need to consider words over $S$ and not over $S \sqcup S^{-1}$, when considering expressions for the elements of $W [\Gamma]$.
A word $\omega \in S^*$ is called \emph{reduced} if $\ell (\omega) = \ell_S (\overline{\omega})$.

Let $\omega$ and $\omega'$ be words over $S$.
We say that $\omega'$ is obtained from $\omega$ by an \emph{M-operation of type I} if there exist words $\mu, \nu \in S^*$ and an element $s \in S$ such that $\omega = \mu  s s \nu$ and $\omega' = \mu \nu$.
We say that $\omega'$ is obtained from $\omega$ by an \emph{M-operation of type II} if there exist $\mu, \nu \in S^*$ and $s,t \in S$, with $s \neq t$ and $m_{s,t} \neq \infty$, such that $\omega = \mu \, \Pi (s,t, m_{s,t}) \, \nu$ and $\omega' = \mu \, \Pi (t,s, m_{s,t}) \, \nu$.
We say that a word $\omega \in S^*$ is \emph{M-reduced} if its length cannot be strictly reduced by any finite sequence of M-operations.

Note that an M-operation of type I strictly decreases word length, whereas an M-operation of type II preserves length.
The latter is reversible, unlike the former.
Furthermore, if $\omega'$ can be obtained from $\omega$ by an M-operation, then $\overline{\omega'} = \overline{\omega}$.

\begin{thm}[Tits \cite{Tit69}]\label{thm2_1}
\begin{itemize}
\item[(1)]
A word $\omega \in S^*$ is reduced if and only if it is M-reduced.
\item[(2)]
Let $\omega$ and $\omega'$ be reduced words over $S$.
If $\overline{\omega} = \overline{\omega'}$, then $\omega$ can be transformed into $\omega'$ by a finite sequence of M-operations of type II.
\end{itemize}
\end{thm}

Let $w \in W [\Gamma]$.
Choose a reduced expression $\omega = s_1 s_2 \cdots s_l$ for $w$, and set $\supp(w) = \{ s_1, s_2, \dots, s_l\}$.
By Theorem \ref{thm2_1}\,(2), this definition does not depend on the choice of the reduced expression.

\begin{prop}\label{prop2_2}
Let $X \subseteq S$ and $w \in W [\Gamma]$.
Then $w \in W_X [\Gamma]$ if and only if $\supp (w) \subseteq X$.
\end{prop}

\begin{proof}
Clearly, if $\supp (w) \subseteq X$, then $w \in W_X [\Gamma]$.
Conversely, let $w \in W_X [\Gamma]$ and let $\omega = s_1 s_2 \cdots s_l$ be a word over $X$ representing $w$.
By Theorem \ref{thm2_1}, there is a finite sequence of M-operations that, applied to $\omega$, yields a reduced word $\omega' = t_1 t_2 \cdots t_k$ with $\overline{\omega'} = \overline{\omega} = w$.
By definition of the M-operations, we have $\supp (w) = \{ t_1, t_2, \dots, t_k \} \subseteq \{ s_1, s_2, \dots, s_l \} \subseteq X$.
\end{proof}

\begin{prop}\label{prop2_3}
Let $X \subseteq S$.
The natural homomorphism $\iota_X : W[\Gamma_X] \to W [\Gamma]$ that maps $x$ to $x$ for all $x \in X$ is injective and preserves word length.
\end{prop}

\begin{proof}
Let $w$ be a non-trivial element of $W[\Gamma_X]$.
Let $\omega$ be a reduced word over $X$ representing $w$.
By Theorem \ref{thm2_1}, $\omega$ is M-reduced, hence it is reduced when viewed as a word over $S$.
It follows that $\ell_X(w) = \ell(\omega) = \ell_S(\iota_X(w)) > 0$ and $\iota_X(w) \neq 1$.
\end{proof}

From now on, for $X \subseteq S$, we identify $W [\Gamma_X]$ with its image $W_X [\Gamma]$ in $W [\Gamma]$ under the homomorphism $\iota_X$.

\begin{prop}\label{prop2_4}
Let $X, Y \subseteq S$. 
Then $W [\Gamma_X] \cap W [\Gamma_Y] = W [\Gamma_{X \cap Y}]$.
\end{prop}

\begin{proof}
The inclusion $W [\Gamma_{X \cap Y}] \subseteq W [\Gamma_X] \cap W [\Gamma_Y]$ is obvious, so we only need to prove the reverse inclusion.
Let $w \in W [\Gamma_X] \cap W [\Gamma_Y]$.
Then, by Proposition \ref{prop2_2}, $\supp(w) \subseteq X \cap Y$, hence $w \in W [\Gamma_{X \cap Y}]$.
\end{proof}

\begin{prop}\label{prop2_5}
Let $X \subseteq S$.
The strong membership problem for $W [\Gamma_X]$ in $W [\Gamma]$ is solvable.
\end{prop}

\begin{proof}
Let $\omega = s_1 s_2 \cdots s_l$ be a word over $S$.
We apply M-operations to $\omega$ until we obtain an M-reduced word $\omega' = t_1 t_2 \cdots t_k$.
This procedure is clearly algorithmic.
Then, $\overline{\omega} \in W [\Gamma_X]$ if and only if $\supp(\overline{\omega}) = \{t_1, t_2, \dots, t_k\} \subseteq X$.
Furthermore, if $\supp(\overline{\omega}) = \{t_1, t_2, \dots, t_k\} \subseteq X$, then $\omega'$ is a word over $X$ representing $\overline{\omega}$.
\end{proof}

\subsection{Standard parabolic subgroups of Artin groups}\label{subsec2_2}

Let $X, Y \subseteq S$ and $w_0 \in W[\Gamma]$.
We say that $w_0$ is \emph{$(X,Y)$-minimal} if it is of minimal length in the double coset $W[\Gamma_X] w_0 W[\Gamma_Y]$.
The following lemma is stated in the exercises of Chapter 4 of \cite{Bou68} (see \cite[Section 4.3]{Dav08} for a proof).

\begin{lem}[Bourbaki \cite{Bou68}]\label{lem2_6}
Let $X, Y \subseteq S$ and $w \in W [\Gamma]$.
There exists a unique $(X,Y)$-minimal element $w_0$ lying in the double coset $W [\Gamma_X] w W [\Gamma_Y]$.
Furthermore, there exist $u \in W [\Gamma_X]$ and $v \in W [\Gamma_Y]$ such that $w = u w_0 v$ and $\ell_S (w) = \ell_S (u) + \ell_S (w_0) + \ell_S (v)$.
\end{lem}

\begin{corl}\label{corl2_7}
{\it Let $X \subseteq S$ and $w \in W [\Gamma]$.
\begin{itemize}
\item[(1)]
$w$ is $(\emptyset,X)$-minimal if and only if $\ell_S (wx) > \ell_S (w)$ for all $x \in X$, and $\ell_S (wx) > \ell_S (w)$ for all $x \in X$ if and only if $\ell_S (wu) = \ell_S (w) + \ell_S (u)$ for all $u \in W [\Gamma_X]$.
\item[(2)]
$w$ is $(X,\emptyset)$-minimal if and only if $\ell_S(xw) > \ell_S(w)$ for all $x \in X$, and $\ell_S(xw) > \ell_S(w)$ for all $x \in X$ if and only if $\ell_S(uw) = \ell_S(u) + \ell_S(w)$ for all $u \in W[\Gamma_X]$.
\end{itemize}}
\end{corl}

Let $X \subseteq S$.
We define a map $\pi_X^* : (S \sqcup S^{-1})^* \to (X \sqcup X^{-1})^*$ as follows.
Let $\omega = s_1^{\varepsilon_1} s_2^{\varepsilon_2} \cdots s_l^{\varepsilon_l} \in (S \sqcup S^{-1})^*$.
We set $u_0=1 \in W [\Gamma]$ and, for $1 \le i \le l$, we set $u_i = s_1 s_2 \cdots s_i \in W [\Gamma]$.
We write each $u_i$ in the form $u_i = v_i w_i$, where $v_i \in W [\Gamma_X]$ and $w_i$ is $(X, \emptyset)$-minimal.
For $1 \le i \le l$ we set 
\[
t_i = \left\{ \begin{array}{ll}
w_{i-1} s_i w_{i-1}^{-1} & \text{if } \varepsilon_i=1\,,\\
w_i s_i w_i^{-1} & \text{if } \varepsilon_i=-1\,.
\end{array} \right.
\]
Then, we set
\[
\tau_i = \left\{ \begin{array}{ll}
1 &\text{if } t_i \not \in X\,,\\
t_i^{\varepsilon_i} &\text{if } t_i \in X\,.
\end{array} \right.
\]
Finally, we define
\[
\pi_X^* (\omega) = \tau_1 \tau_2 \cdots \tau_l \in (X \sqcup X^{-1})^*\,.
\]

\begin{rem}\label{rem2_8}
\begin{itemize}
\item[(1)] 
We have $\pi_X^* (\omega) = \omega$ for any $\omega \in (X \sqcup X^{-1})^*$.
\item[(2)] 
The word $\pi_X^* (\omega)$ can be computed algorithmically for any $\omega \in (S \sqcup S^{-1})^*$ using the results from Subsection \ref{subsec2_1}.
\item[(3)] 
We have $\ell (\pi_X^* (\omega)) \le \ell (\omega)$ for all $\omega \in (S \sqcup S^{-1})^*$.
\end{itemize}
\end{rem}

The following is proven in \cite[Proposition 2.3]{BP23} (see also \cite[Proposition 0.2]{God23}).

\begin{thm}[Blufstein--Paris \cite{BP23}]\label{thm2_9}
Let $X \subseteq S$.
Let $\omega, \mu \in (S \sqcup S^{-1})^*$ be words representing the same element in $A [\Gamma]$.
Then $\pi_X^* (\omega)$ and $\pi_X^* (\mu)$ represent the same element in $A [\Gamma_X]$.
In other words, the map $\pi_X^* : (S \sqcup S^{-1})^* \to (X \sqcup X^{-1})^*$ induces a map $\pi_X : A[\Gamma] \to A [\Gamma_X]$.
\end{thm}

\begin{rem}\label{rem2_10}
The map $\pi_X$ from Theorem \ref{thm2_9} is not, in general, a homomorphism.
\end{rem}

\begin{prop}[Van der Lek \cite{vLek83}, Charney--Paris \cite{CP14}]\label{prop2_11}
Let $X \subseteq S$.
The natural homomorphism $\iota_X : A [\Gamma_X] \to A [\Gamma]$ that maps $x$ to $x$ for all $x \in X$ is injective and preserves word length.
\end{prop}

\begin{proof}
From Remark \ref{rem2_8}\,(1), it follows that $\pi_X \circ \iota_X = \id_{A [\Gamma_X]}$, hence $\iota_X$ is injective.
Let $\ell_S : A [\Gamma] \to \N$ denote the word length with respect to $S$.
Let $g \in A [\Gamma_X]$.
Since $X \subseteq S$, we have $\ell_X (g) \ge \ell_S (\iota_X (g))$.
Let $\omega$ be a word over $S \sqcup S^{-1}$ representing $\iota_X(g)$ such that $\ell_S (\iota_X (g)) = \ell (\omega)$.
The word $\pi_X^* (\omega)$ represents $g = \pi_X (\iota_X(g))$, hence, by Remark \ref{rem2_8}\,(3),
\[
\ell_S (\iota_X(g)) = \ell (\omega) \ge \ell (\pi_X^* (\omega)) \ge \ell_X (g)\,.
\]
So, $\ell_X (g) = \ell_S(\iota_X (g))$.
\end{proof}

From now on, for $X \subseteq S$, we identify $A [\Gamma_X]$ with its image $A_X [\Gamma]$ in $A [\Gamma]$ under the homomorphism $\iota_X$.

\begin{prop}[Van der Lek \cite{vLek83}]\label{prop2_12}
Let $X, Y \subseteq S$. 
Then $A [\Gamma_X] \cap A [\Gamma_Y] = A [\Gamma_{X \cap Y}]$.
\end{prop}

\begin{proof}
The inclusion $A [\Gamma_{X \cap Y}] \subseteq A [\Gamma_X] \cap A [\Gamma_Y]$ is obvious, hence we only need to prove the reverse inclusion. 
Let $g \in A [\Gamma_X] \cap A [\Gamma_Y]$.
We choose a word $\omega = s_1^{\varepsilon_1} \cdots s_n^{\varepsilon_n} \in (Y \sqcup Y^{-1})^*$ such that $\overline{\omega} = g$.
Since $\omega \in (Y \sqcup Y^{-1})^*$, in the definition of $\pi_X^* (\omega)$ given above, we see that $u_i \in W [\Gamma_Y]$, which implies $v_i, w_i \in W [\Gamma_Y]$, for all $0 \le i \le l$, and $s_i \in Y$ for all $1 \le i \le l$.
This implies that either $\tau_i \in (Y \sqcup Y^{-1})$ or $\tau_i=1$ for all $1 \le i \le l$, hence $\pi_X^* (\omega) \in (Y \sqcup Y^{-1})^*$.
As we also have $\pi_X^* (\omega) \in (X \sqcup X^{-1})^*$, it follows that $\pi_X^* (\omega) \in ((X \cap Y) \sqcup (X \cap Y)^{-1})^*$.
Since $g \in A [\Gamma_X]$, we conclude that $g = \pi_X (g) = \overline{\pi_X^* (\omega)} \in A [\Gamma_{X \cap Y}]$.
\end{proof}

\begin{prop}\label{prop2_13}
Suppose that the word problem in $A [\Gamma]$ is solvable.
Let $X \subseteq S$.
Then the strong membership problem for $A [\Gamma_X]$ in $A [\Gamma]$ is solvable.
\end{prop}

\begin{proof}
Let $\omega \in (S \sqcup S^{-1})^*$.
Let $g = \overline{\omega}$ and let $\mu = \pi_X^* (\omega) \in (X \sqcup X^{-1})^*$.
As noted in Remark \ref{rem2_8}\,(2), the word $\mu$ can be computed algorithmically.
Furthermore, we have $g \in A [\Gamma_X]$ if and only if$g = \pi_X (g)$, that is, $\overline{\omega} = \overline{\mu}$.
Let $\nu = \omega \mu^{-1}$.
Using the solution to the word problem in $A [\Gamma]$, we can decide whether $\overline{\nu} = 1$.
If $\overline{\nu} \neq 1$, then $g \neq \pi_X (g)$, hence $g \not \in A [\Gamma_X]$.
If $\overline{\nu} = 1$, then $g = \pi_X (g) \in A [\Gamma_X]$, and $\mu$ is a word over $X \sqcup X^{-1}$ representing $g$.
\end{proof}

\subsection{Virtual Artin groups}\label{subsec2_3}

There is a homomorphism $\pi_K : \VA [\Gamma] \to W [\Gamma]$ defined by
\[
\pi_K (\sigma_s) = 1 \text{ and } \pi_K (\tau_s) = s\,, \text{ for all } s \in S\,.
\]
The kernel of $\pi_K$ is called the \emph{kure virtual Artin group} of $\Gamma$ and is denoted by $\KVA [\Gamma]$.
Furthermore, there is a homomorphism $\iota_W : W [\Gamma] \to \VA [\Gamma]$ that maps  $s$ to $\tau_s$ for all $s \in S$.
It is easily seen that $\pi_K \circ \iota_W = \id_{W [\Gamma]}$, hence $\iota_W$ is injective, $\pi_K$ is surjective, and we have the semidirect product decomposition $\VA [\Gamma] = \KVA [\Gamma] \rtimes W [\Gamma]$.

Let $\Pi = \{ \alpha_s \mid s \in S \}$ be a set in one-to-one correspondence with $S$.
Let $V$ be the real vector space with $\Pi$ as a basis, and let $\langle \cdot , \cdot \rangle : V \times V \to \R$ be the symmetric bilinear form defined by
\[
\langle \alpha_s, \alpha_t \rangle = \left\{ \begin{array}{ll}
-2\, \cos(\pi/m_{s,t}) & \text{if } m_{s,t} \neq \infty\,,\\
-2 &\text{if } m_{s,t} = \infty\,.
\end{array} \right.
\]
There is a faithful linear representation $\rho : W \hookrightarrow \GL (V)$, called the \emph{canonical representation}, which preserves the bilinear form $\langle \cdot, \cdot \rangle$, defined by
\[
\rho(s) (v) = v - \langle v, \alpha_s \rangle \, \alpha_s\,, \quad \text{for } v \in V\,,\ s \in S\,.
\]
From now on, we assume that $W [\Gamma]$ is embedded into $\GL (V)$ via the representation $\rho$ and, for $w \in W [\Gamma]$ and $v \in V$, we write $w (v)$ in place of $\rho(w) (v)$.

The set $\Phi [\Gamma] = \{ w (\alpha_s) \mid w \in W [\Gamma]\,,\ s \in S \}$ is called the \emph{root system} of $\Gamma$.
A root $\beta \in \Phi [\Gamma]$ is said to be \emph{positive} (resp. \emph{negative}) if it can be written as $\beta = \sum_{s \in S} \lambda_s \alpha_s$, where $\lambda_s \ge 0$ for all $s \in S$ (resp. $\lambda_s \le 0$ for all $s \in S$).
We denote by $\Phi^+ [\Gamma]$ (resp. $\Phi^- [\Gamma]$) the set of positive roots (resp. negative roots). 
We know from \cite{Deo82} that $\Phi [\Gamma] = \Phi^+ [\Gamma] \sqcup \Phi^{-} [\Gamma]$ and $\Phi^- [\Gamma] = \{ - \beta \mid \beta \in \Phi^+ [\Gamma]\}$.

We define a Coxeter matrix $\hat M = (\hat m_{\beta, \gamma})_{\beta, \gamma \in \Phi [\Gamma]}$ over $\Phi [\Gamma]$ as follows.
\begin{itemize}
\item[(a)]
We set $\hat m_{\beta, \beta} = 1$ for all $\beta \in \Phi [\Gamma]$.
\item[(b)]
Let $\beta, \gamma \in \Phi [\Gamma]$ with $\beta \neq \gamma$.
Suppose there exist $s,t \in S$ and $w \in W [\Gamma]$ such that $w (\alpha_s) = \beta$, $w (\alpha_t) = \gamma$, and $m_{s,t} \neq \infty$.
Then we set $\hat m_{\beta, \gamma} = m_{s,t}$.
\item[(c)]
We set $\hat m_{\beta, \gamma} = \infty$ in the other cases.
\end{itemize}
Note that the definition of $\hat m_{\beta, \gamma}$ does not depend on the choice of $s,t \in S$ and $w \in W [\Gamma]$ in Case (b) (see \cite{BPT23}).
Following \cite{BPT23}, we denote by $\hat \Gamma$ the Coxeter graph of $\hat M$ and by $\{ \hat \delta_\beta \mid \beta \in \Phi [\Gamma] \}$ the standard generating set for $A [\hat \Gamma]$.

\begin{rem}\label{rem2_14}
It is known from \cite{Deo82} that $\Phi [\Gamma]$ is finite if and only if $W [\Gamma]$ is finite.
In particular, if $W [\Gamma]$ is infinite, then $\hat \Gamma$ is a Coxeter graph with infinitely many vertices.
However, as noted in \cite{BPT23}, the results on finitely generated Artin groups that we use can be easily extended to infinitely generated Artin groups.
\end{rem}

Let $\beta \in \Phi [\Gamma]$.
We choose $w \in W [\Gamma]$ and $s \in S$ such that $\beta = w (\alpha_s)$, and we set  $\delta_\beta = \iota_W (w)\, \sigma_s\, \iota_W (w)^{-1} \in \KVA [\Gamma]$.
According to \cite[Lemma 2.2]{BPT23}, this definition does not depend on the choice of $w$ and $s$.
The Coxeter graph $\hat \Gamma$ and the group $\KVA [\Gamma]$ are related by the following result proved in \cite[Theorem 2.3]{BPT23}.

\begin{thm}[Bellingeri--Paris--Thiel \cite{BPT23}]\label{thm2_15}
The map $\{ \hat \delta_\beta \mid \beta \in \Phi [\Gamma] \} \to \{ \delta_\beta \mid \beta \in \Phi [\Gamma]\}$, $\hat \delta_\beta \mapsto \delta_\beta$, induces an isomorphism $\varphi: A [\hat \Gamma] \to \KVA [\Gamma]$.
\end{thm}

From now on we identify $A [\hat \Gamma]$ with its image $\KVA [\Gamma]$ in $\VA [\Gamma]$ under the homomorphism $\varphi$.

\subsection{Amalgamated products}\label{subsec2_4}

Let $G$ be a group and $H$ be a subgroup of $G$.
A \emph{transversal} of $H$ in $G$ is a subset $T \subseteq G$ such that for every $g \in G$, there exists a unique $\theta \in T$ such that $g H = \theta H$.
For convenience, we will always assume that $1 \in T$.
The following result is classical and is proven in \cite[Section 1.1, Théorème 1]{Ser77}.

\begin{thm}[Serre \cite{Ser77}]\label{thm2_16}
Let $G_1$, $G_2$, and $H$ be groups.
Suppose that $H$ is a subgroup of both $G_1$ and $G_2$, and set $G = G_1 *_H G_2$.
For each $j \in \{1,2\}$, we choose a transversal $T_j$ of $H$ in $G_j$.
Then every element $g \in G$ can be uniquely written in the form $g = \theta_1 \theta_2 \cdots \theta_l h$ where:
\begin{itemize}
\item[(a)]
$h \in H$ and, for each $1 \le i \le l$, there exists $j (i) \in \{1,2\}$ such that $\theta_i \in T_{j (i)} \setminus \{1\}$;
\item[(b)]
$j(i) \neq j(i+1)$ for all $1 \le i \le l-1$.
\end{itemize}
In particular, $g \in H$ if and only if $l=0$.
\end{thm}

The expression for $g$ given in the above theorem is called the \emph{normal form} of $g$.
Note that it depends on the choice of transversals of $H$ in $G_1$ and in $G_2$.

In our study, we use the following corollary.
While this result is standard, a proof is included for the sake of completeness.

\begin{corl}\label{corl2_17}
Let $G_1$, $G_2$, and $H$ be groups.
Suppose that $H$ is a subgroup of $G_1$ and $G_2$, and set $G = G_1 *_H G_2$.
Let $g \in G$.
Suppose $g$ can be written in the form $g = g_1 g_2 \cdots g_l$ where:
\begin{itemize}
\item[(a)]
$l \ge 1$,
\item[(b)]
for all $1 \le i \le l$ there exists $j(i) \in \{1,2\}$ such that $g_i \in G_{j(i)} \setminus H$,
\item[(c)]
$j(i) \neq j(i+1)$ for all $1 \le i \le l-1$.
\end{itemize}
Then $g \not \in H$. In particular, $g \neq 1$.
\end{corl}

\begin{proof}
For each $j \in \{1,2\}$ we choose a transversal $T_j$ of $H$ in $G_j$, and we consider normal forms with respect to this choice of transversals.
For $1 \le i \le l$, we define $\theta_i \in T_{j(i)} \setminus \{1\}$ and $h_i \in H$ by induction on $i$ as follows.
The element $\theta_1$ is the unique element of $T_{j(1)}$ such that $g_1 H = \theta_1 H$.
Since $g_1 \not \in H$, we have $\theta_1 \neq 1$.
We then set $h_1 = \theta_1^{-1} g_1$.
Assume that $i \ge 2$ and that $h_k$ and $\theta_k$ are defined for $k <i$. 
The element $\theta_i$ is the unique element of $T_{j(i)}$ such that $h_{i-1} g_i H = \theta_i H$.
Since $h_{i-1} g_i \not \in H$ (as $g_i \not \in H$ and $h_{i-1} \in H$), we have $\theta_i \neq 1$.
We then set $h_i = \theta_i^{-1} h_{i-1} g_i$.
By construction, the expression $\theta_1 \theta_2 \cdots \theta_l h_l$ is the normal form of $g$.
Since we are under the assumption that $l \ge 1$, by Theorem \ref{thm2_16} it follows that $g \not \in H$.
\end{proof}

The following proposition is the main tool for proving Theorem \ref{thm1_4} from Theorem \ref{thm1_6}.

\begin{prop}\label{prop2_18}
Let $G_1$, $G_2$, and $H$ be finitely generated groups.
Suppose that $H$ is a subgroup of $G_1$ and $G_2$, and set $G = G_1 *_H G_2$.
If the word problem in $G_j$ is solvable for each $j \in \{1,2\}$, and the strong membership problem for $H$ in $G_j$ is solvable for each $j \in \{1,2\}$, then the word problem in $G$ is solvable.
\end{prop}

\begin{proof}
As ever, if $G$ is a group generated by a set $S$, and $\omega$ is a word over $S \sqcup S^{-1}$, we denote by $\overline{\omega}$ the element of $G$ represented by $\omega$.
We are given finite generating sets $S_1$, $S_2$, and $S_H$ of $G_1$, $G_2$, and $H$, respectively.
Furthermore, for each $j \in \{1,2\}$ and each generator $y \in S_H$, we are given a word $\iota_j (y) \in (S_j \sqcup S_j^{-1})^*$ representing $y$ in $G_j$.
We have to show that there exists an algorithm which, given a word $\omega$ over $S_1 \sqcup S_1^{-1} \sqcup S_2 \sqcup S_2^{-1}$, decides whether $\overline{\omega} = 1$ in $G$.
Let $\omega$ be a word over $S_1 \sqcup S_1^{-1} \sqcup S_2 \sqcup S_2^{-1}$.
This can easily be written as a finite product $\omega = \omega_1 \omega_2 \cdots \omega_n$ where, for each $1 \le i \le n$, there exists $j(i) \in \{1,2\}$ such that $\omega_i \in (S_{j(i)} \sqcup S_{j(i)}^{-1})^*$.
We proceed by induction on $n$.

If $n=0$, then $\overline{\omega} = 1$ and the algorithm stops.
Suppose $n=1$.
Then we apply the solution to the word problem in $G_{j(1)}$ to decide whether $\overline{\omega} = \overline{\omega_1} = 1$.

Assume that $n \ge 2$ and the induction hypothesis holds. 
Suppose there exists $i \in \{1, \dots, n-1\}$ such that $j(i) = j(i+1)$.
Let $\omega_i' = \omega_i \omega_{i+1} \in (S_j \sqcup S_j^{-1})^*$, where $j = j(i) = j(i+1)$.
Then 
\[
\omega = \omega_1 \cdots \omega_{i-1} \omega_i' \omega_{i+2} \cdots \omega_n\,,
\]
hence, by the induction hypothesis, we can decide whether $\overline{\omega} = 1$.

So, we can assume that $j(i) \neq j(i + 1)$ for all $1 \le i \le n-1$.
By hypothesis, there exists an algorithm to decide whether $\overline{\omega_i}$ belongs to $H$.
If $\overline{\omega_i} \not \in H$ for all $i \in \{1, \dots, n\}$, then, by Corollary \ref{corl2_17}, 
\[
\overline{\omega} = \overline{\omega_1} \, \overline{\omega_2} \cdots \overline{\omega_n} \not \in H\,,
\]
hence $\overline{\omega} \neq 1$.

So, we can assume that there exists $i \in \{1, \dots, n\}$ such that $\overline{\omega_i} \in H$.
Without loss of generality, assume $j(i)=1$.
Since the strong membership problem for $H$ in $G_1$ is solvable, there exists an algorithm which determines a word $\mu_i \in (S_H \sqcup S_H^{-1})^*$ such that $\overline{\mu_i} = \overline{\omega_i}$.
Let $\iota_2^* : (S_H \sqcup S_H^{-1})^* \to (S_2 \sqcup S_2^{-1})^*$ be the monoid homomorphism which sends $y^{\pm 1}$ to $\iota_2 (y)^{\pm 1}$ for all $y \in S_H$.
Let $\nu_i =\iota_2^* (\mu_i)$.
Then $\nu_i \in (S_2 \sqcup S_2^{-1})^*$ and $\overline{\omega_i} = \overline{\nu_i}$.
Let
\[
\omega' = \omega_1 \cdots \omega_{i-1} \nu_i \omega_{i+1} \cdots \omega_n\,.
\]
Note that $\overline{\omega} = \overline{\omega'}$, $j(i-1) = 2$ (if $i \ge 2$), and $j(i+1) = 2$ (if $i \le n-1$).
By applying to $\omega'$ the same argument as the one given two paragraphs above, we can with the induction hypothesis decide whether $\overline{\omega} = \overline{\omega'} = 1$. 
\end{proof}


\section{Standard parabolic subgroups of virtual Artin groups}\label{sec3}

The set $\RR[\Gamma] = \{ w s w^{-1} \mid s \in S\,,\ w \in W[\Gamma] \}$ is called the set of \emph{reflections} of $W[\Gamma]$.
To every root $\beta \in \Phi[\Gamma]$, we associate the reflection $r_\beta = w s w^{-1}$, where $w \in W[\Gamma]$ and $s \in S$ are such that $\beta = w (\alpha_s)$.
It is known from \cite{Deo82} that this definition does not depend on the choice of $w$ and $s$.
Furthermore, for any $\beta, \gamma \in \Phi[\Gamma]$, we have $r_\beta = r_\gamma$ if and only if $\gamma = \pm \beta$.
In our proofs, we use the following two lemmas, which are standard results in the theory of Coxeter groups.
Lemma \ref{lem3_1} is proved in \cite[Proposition 2.2]{Deo82}, and Lemma \ref{lem3_2} is proved in \cite[Lemma 1.7]{BH93}.

\begin{lem}[Deodhar \cite{Deo82}]\label{lem3_1}
Let $w \in W [\Gamma]$ and $s \in S$. 
Then
\begin{gather*}
\ell_S (ws) > \ell_S (w) \ \Leftrightarrow\ w(\alpha_s) \in \Phi^+ [\Gamma]\,,\\
\ell_S (ws) < \ell_S (w) \ \Leftrightarrow\ w(\alpha_s) \in \Phi^- [\Gamma]\,.
\end{gather*}
\end{lem}

The \emph{depth} of a positive root $\beta \in \Phi^+ [\Gamma]$ is defined to be
\[
\dpt (\beta) = \min \{l \in \N \mid \exists w \in W [\Gamma] \text{ such that } w (\beta) \in \Phi^- [\Gamma] \text{ and } \ell_S (w) = l\}\,.
\]

\begin{lem}[Brink--Howlett \cite{BH93}]\label{lem3_2}
Let $s \in S$ and $\beta \in \Phi^+ [\Gamma] \setminus \{ \alpha_s\}$. 
Then 
\[
\dpt (s (\beta)) = \left\{ \begin{array}{ll}
\dpt (\beta) -1 & \text{if } \langle \beta, \alpha_s \rangle >0\,,\\
\dpt (\beta)  & \text{if } \langle \beta, \alpha_s \rangle = 0\,,\\
\dpt (\beta) + 1 & \text{if } \langle \beta, \alpha_s \rangle < 0\,.
\end{array} \right.
\]
\end{lem}

\begin{rem}\label{rem3_3}
Let $\beta \in \Phi^+ [\Gamma]$.
If $\dpt (\beta) = 1$, then there exists $s \in S$ such that $\beta = \alpha_s$.
If $\dpt (\beta) \ge 2$, then there exists $s \in S$ such that $\dpt (s (\beta)) < \dpt (\beta)$.
Indeed, if we write $\beta = w (\alpha_t)$ with $t \in S$, $w \in W [\Gamma]$, and $w$ of minimal length, then $\dpt (\beta) = \ell_S(w) + 1$.
So, if $s \in S$ is such that $\ell_S (sw) < \ell_S (w)$, then $\dpt (s (\beta)) < \dpt (\beta)$.
\end{rem}

\begin{lem}\label{lem3_4}
Let $X, Y \subseteq S$ and $w \in W [\Gamma]$.
Let $w_0$ be the unique $(X,Y)$-minimal element in the double coset $W[\Gamma_X] w W [\Gamma_Y]$.
Suppose that $w W [\Gamma_Y] w^{-1} \subseteq W [\Gamma_X]$.
Then $w_0 Y w_0^{-1} \subseteq X$ and there exists $u \in W [\Gamma_X]$ such that $w = u w_0$.
\end{lem}

\begin{proof}
Let $a \in W [\Gamma_X]$ and $b \in W [\Gamma_Y]$ be such that $w = a w_0 b$.
We have
\[
(a w_0) W [\Gamma_Y] (a w_0)^{-1} = (a w_0 b) W [\Gamma_Y] (a w_0 b)^{-1} = w W [\Gamma_Y] w^{-1} \subseteq W [\Gamma_X]\,,
\]
hence
\[
w_0 W [\Gamma_Y] w_0^{-1} \subseteq a^{-1} W [\Gamma_X] a = W [\Gamma_X]\,.
\]
Let $y \in Y$.
Let $x \in W [\Gamma_X]$ be such that $w_0 y w_0^{-1} = x$. 
So, $w_0 y = x w_0$.
Note that, since $w_0$ is $(X,Y)$-minimal, $w_0$ is both $(X, \emptyset)$-minimal and $(\emptyset, Y)$-minimal.
By Corollary \ref{corl2_7}, it follows that
\[
1 + \ell_S (w_0) = \ell_S (w_0 y) = \ell_S (x w_0) = \ell_S (x) + \ell_S (w_0)\,,
\]
hence, $\ell_S (x) = 1$, that is, $x = w_0 y w_0^{-1} \in X$.
This proves that $w_0 Y w_0^{-1} \subseteq X$.

Let $b' = w_0 b w_0^{-1}$ and $u = ab'$.
Note that $b' \in W [\Gamma_X]$, because $w_0 W [\Gamma_Y] w_0^{-1} \subseteq W [\Gamma_X]$, hence $u = a b' \in W[\Gamma_X]$.
Now, $w = a w_0 b = ab' w_0 = u w_0$.
\end{proof}

Let $X \subseteq S$.
We define the subspace
\[
V_X = \bigoplus_{x \in X} \R \alpha_x \subseteq V = \bigoplus_{s \in S} \R \alpha_s\,,
\]
and we consider the root system $\Phi [\Gamma_X]$ as a subset of $V_X \subseteq V$.
Although the following lemma is proved in \cite{Qi07} using a similar argument, we provide a proof here in order to make the presentation more accessible and self-contained.

\begin{lem}\label{lem3_5}
Let $X \subseteq S$. 
Then $\Phi [\Gamma_X] = V_X \cap \Phi [\Gamma]$.
\end{lem}

\begin{proof}
The inclusion $\Phi [\Gamma_X] \subseteq V_X \cap \Phi [\Gamma]$ is obvious, hence we only need to prove the reverse inclusion.
Let $\beta \in V_X \cap \Phi [\Gamma]$.
Up to replacing $\beta$ by $- \beta$ if necessary, we can assume that $\beta \in \Phi^+ [\Gamma]$.
We prove that $\beta \in \Phi [\Gamma_X]$ by induction on $\dpt (\beta)$.
If $\dpt(\beta) = 1$, then there exists $s \in S$ such that $\beta = \alpha_s$.
Since $\beta \in V_X$, we must have $s \in X$, which implies $\beta \in \Phi[\Gamma_X]$.
Now, suppose that $\dpt(\beta) \ge 2$, and that the induction hypothesis holds. 
Let $\beta = \sum_{x \in X} \lambda_x \alpha_x$ be the representation of $\beta$ in the basis $\Pi_X = \{ \alpha_x \mid x \in X \}$ of $V_X$.
According to Remark \ref{rem3_3}, there exists $s \in S$ such that $\dpt(s(\beta)) < \dpt(\beta)$.
By Lemma \ref{lem3_2}, this implies $\langle \beta, \alpha_s \rangle >0$.
Suppose  that $s \not \in X$.
Since $\lambda_x \ge 0$ and $\langle \alpha_x, \alpha_s \rangle \le 0$ for all $x \in X$, we would have $\langle \beta, \alpha_s \rangle \le 0$, which is a contradiction.
So, $s \in X$.
It follows that $\beta' = s (\beta) \in V_X \cap \Phi [\Gamma]$.
By the induction hypothesis, $\beta' \in \Phi [\Gamma_X]$, hence $\beta = s (\beta') \in \Phi [\Gamma_X]$.
\end{proof}

In our study, the elements of $\Phi [\Gamma]$ are viewed as linear combinations of the elements of $\Pi = \{ \alpha_s \mid s \in S\}$.
This makes sense from an algorithmic point of view, as we can take the coefficients in a number field $\K$ containing $\{\cos (\pi/m_{s,t}) \mid s,t \in S\,,\ s \neq t \text{ and } m_{s,t} \neq \infty\}$.
If $\beta = w (\alpha_s)$ with $w \in W [\Gamma]$ and $s \in S$, then the coordinates of $\beta$ are easily determined by applying the canonical linear representation.
The reverse operation which, given $\beta \in \Phi [\Gamma]$, consists of finding $s \in S$ and $w \in W [\Gamma]$ such that $\beta = w (\alpha_s)$, is less obvious (see Lemma \ref{lem4_2} for a solution).

\begin{corl}\label{corl3_6}
Let $X \subseteq S$.
There exists an algorithm that, given $\beta \in \Phi [\Gamma]$, decides whether $\beta \in \Phi [\Gamma_X]$ or not.
\end{corl}

\begin{proof}
Let $\beta = \sum_{s \in S} \lambda_s \alpha_s$ be a root in $\Phi [\Gamma]$.
Then $\beta \in \Phi [\Gamma_X]$ if and only if $\lambda_s = 0$ for all $s \in S \setminus X$.
\end{proof}

\begin{lem}\label{lem3_7}
Let $X \subseteq S$.
Let $\beta, \gamma \in \Phi [\Gamma_X]$.
Suppose there exist $s,t \in S$ and $w \in W [\Gamma]$ such that $\beta = w(\alpha_s)$ and $\gamma = w(\alpha_t)$.
Then there exist $x,y \in X$ and $u \in W [\Gamma_X]$ such that $\beta = u(\alpha_x)$ and $\gamma = u(\alpha_y)$.
\end{lem}

\begin{proof}
Let $Y = \{s,t\}$.
We have $wsw^{-1} = r_\beta \in W [\Gamma_X]$ and $wtw^{-1} = r_\gamma \in W[\Gamma_X]$, hence $w W[\Gamma_Y] w^{-1} \subseteq W [\Gamma_X]$.
Let $w_0$ denote the unique $(X,Y)$-minimal element in the double coset $W [\Gamma_X] w W[\Gamma_Y]$.
By Lemma \ref{lem3_4}, we have $w_0 Y w_0^{-1} \subseteq X$, and there exists $u \in W[\Gamma_X]$ such that $w = u w_0$.
Let $x,y \in X$ be such that $w_0 s w_0^{-1} = x$ and $w_0 t w_0^{-1} = y$.
As noted at the beginning of this section, a given reflection is defined by exactly two roots opposite to each other, hence the equality $w_0 s w_0^{-1} = x$ implies that $w_0 (\alpha_s) = \alpha_x$ or $w_0 (\alpha_s) = - \alpha_x$.
Since $w_0$ is $(\emptyset, Y)$-minimal, by Corollary \ref{corl2_7}, $\ell_S (w_0 s) > \ell_S (w_0)$, hence, by Lemma \ref{lem3_1}, $w_0 (\alpha_s) \in \Phi^+ [\Gamma]$, and therefore $w_0 (\alpha_s) = \alpha_x$.
Similarly, $w_0 (\alpha_t) = \alpha_y$.
Finally,
\[
\beta = w(\alpha_s) = u w_0 (\alpha_s) = u (\alpha_x) \quad \text{and} \quad \gamma = w (\alpha_t) = u w_0 (\alpha_t) = u (\alpha_y)\,.
\proved
\]
\end{proof}

We now establish some notation for what follows.
Recall that $\hat \Gamma$ denotes the Coxeter graph from Theorem \ref{thm2_15}, and that $\hat M = (\hat m_{\beta, \gamma})_{\beta, \gamma \in \Phi [\Gamma]}$ denotes its Coxeter matrix.
Recall also that we identify $A [\hat \Gamma]$ with $\KVA [\Gamma]$.
Let $X \subseteq S$.
We denote by $\hat \Gamma_{\Phi[\Gamma_X]}$ the full Coxeter subgraph of $\hat \Gamma$ spanned by $\Phi [\Gamma_X]$.
By Proposition \ref{prop2_11}, $A [\hat \Gamma_{\Phi [\Gamma_X]}]$ is a subgroup of $A [\hat \Gamma]$.
On the other hand, let $\widehat{\Gamma_X}$ be the Coxeter graph obtained by applying Theorem \ref{thm2_15} to $\VA[\Gamma_X]$.
Again, we identify $A[\widehat{\Gamma_X}]$ with $\KVA[\Gamma_X]$.
The two Coxeter graphs, $\hat \Gamma_{\Phi[\Gamma_X]}$ and $\widehat{\Gamma_X}$, share the same vertex set, $\Phi[\Gamma_X]$, but they are not isomorphic {\it a priori}.
We now show that they are, in fact, isomorphic.

\begin{lem}\label{lem3_8}
Let $X \subseteq S$.
The identity map on $\Phi [\Gamma_X]$ induces an isomorphism of Coxeter graphs between $\widehat{\Gamma_X}$ and $\hat \Gamma_{\Phi [\Gamma_X]}$.
\end{lem}

\begin{proof}
Let $\hat M^X = (\hat m_{\beta, \gamma}^X)_{\beta, \gamma \in \Phi [\Gamma_X]}$ be the Coxeter matrix of $\widehat{\Gamma_X}$.
We must show that $\hat m_{\beta, \gamma} = \hat m^X_{\beta, \gamma}$ for all $\beta, \gamma \in \Phi [\Gamma_X]$.
To this end, it suffices to prove the following two claims:
\begin{itemize}
\item[(1)]
Let $\beta, \gamma \in \Phi [\Gamma_X]$ with $\beta \neq \gamma$.
If $\hat m^X_{\beta, \gamma} \neq \infty$, then $\hat m_{\beta, \gamma} \neq \infty$ and $\hat m_{\beta, \gamma} = \hat m^X_{\beta, \gamma}$.
\item[(2)]
Let $\beta, \gamma \in \Phi [\Gamma_X]$ with $\beta \neq \gamma$.
If $\hat m_{\beta, \gamma} \neq \infty$, then $\hat m_{\beta, \gamma}^X \neq \infty$ and $\hat m_{\beta, \gamma}^X = \hat m_{\beta, \gamma}$.
\end{itemize}
Let $\beta, \gamma \in \Phi [\Gamma_X]$ with $\beta \neq \gamma$.
Suppose first that $\hat m^X_{\beta, \gamma} \neq \infty$.
There exist $x,y \in X$ and $u \in W [\Gamma_X]$ such that $\beta = u (\alpha_x)$, $\gamma = u (\alpha_y)$, $m_{x,y} \neq \infty$, and $\hat m_{\beta, \gamma}^X = m_{x,y}$.
Then, by definition, $\hat m_{\beta, \gamma} \neq \infty$ and $\hat m_{\beta, \gamma} = m_{x,y} = \hat m_{\beta, \gamma}^X$.
Suppose now that $\hat m_{\beta, \gamma} \neq \infty$.
There exist $s, t \in S$ and $w \in W [\Gamma]$ such that $\beta = w (\alpha_s)$, $\gamma = w (\alpha_t)$, $m_{s,t} \neq \infty$, and $\hat m_{\beta, \gamma} = m_{s,t}$.
By Lemma \ref{lem3_7}, there exist $x,y \in X$ and $u \in W [\Gamma_X]$ such that $\beta = u(\alpha_x)$ and $\gamma = u (\alpha_y)$.
Note that $m_{x,y}$ and $m_{s,t}$ are both equal to the order of $r_\beta r_\gamma = w(st) w^{-1} = u(xy) u^{-1}$ in $W[\Gamma]$, hence $m_{x,y} \neq \infty$ and $m_{x,y} = m_{s,t}$.
Finally, by definition, $\hat m_{\beta, \gamma}^X \neq \infty$ and $\hat m_{\beta, \gamma}^X = m_{x,y} = \hat m_{\beta, \gamma}$.
\end{proof}

In the proofs of Theorems \ref{thm1_1} and \ref{thm1_2} below, we use classical results of Van der Lek \cite{vLek83} (see Propositions \ref{prop2_3} and \ref{prop2_4}). 
Recall that, although these results are proved for finitely generated Artin groups, their extension to infinitely generated Artin groups is immediate via inductive limits of finitely generated standard parabolic subgroups (see \cite{BPT23}).

\begin{proof}[Proof of Theorem \ref{thm1_1}]
Let $X \subseteq S$.
Let $\varphi : \VA [\Gamma_X] \to \VA [\Gamma]$ be the natural homomorphism that maps $\sigma_x$ to $\sigma_x$ and $\tau_x$ to $\tau_x$ for all $x \in X$.
We have the following commutative diagram:
\[
\xymatrix{
1 \ar[r] & \KVA [\Gamma_X] \ar[r] \ar[d]_{\psi} & \VA [\Gamma_X] \ar[r] \ar[d]_{\varphi} & W [\Gamma_X] \ar[r] \ar[d]_{\iota_X} & 1 \\
1 \ar[r] & \KVA [\Gamma] \ar[r] & \VA [\Gamma] \ar[r] & W [\Gamma] \ar[r] & 1
}
\]
where the two rows are exact sequences, $\psi : \KVA [\Gamma_X] \to \KVA [\Gamma]$ is the restriction of $\varphi$ to $\KVA [\Gamma_X]$, and $\iota_X : W [\Gamma_X] \to W [\Gamma]$ is the natural homomorphism that maps $x$ to $x$ for all $x \in X$.
We know from Proposition \ref{prop2_3} that $\iota_X$ is injective.
By Theorem \ref{thm2_15}, we can assume that $\KVA[\Gamma_X] = A[\widehat{\Gamma_X}]$ and $\KVA[\Gamma] = A[\hat \Gamma]$.
Furthermore, by Lemma \ref{lem3_8}, we can assume that $A[\widehat{\Gamma_X}] = A[\hat \Gamma_{\Phi[\Gamma_X]}]$.
Once these identifications are made, $\psi$ becomes the natural homomorphism $\psi : A [\hat \Gamma_{\Phi [\Gamma_X]}] \to A [\hat \Gamma]$ that  maps $\delta_\beta$ to $\delta_\beta$ for all $\beta \in \Phi [\Gamma_X]$.
We know from Proposition \ref{prop2_11} that this homomorphism is injective.
Finally, by the five lemma, we conclude that $\varphi : \VA [\Gamma_X] \to \VA [\Gamma]$ is injective.
\end{proof}

From now on, for $X \subseteq S$, we identify $\VA [\Gamma_X]$ with its image $\VA_X [\Gamma]$ in $\VA [\Gamma]$ under the above homomorphism $\varphi$.

\begin{proof}[Proof of Theorem \ref{thm1_2}]
Let $X, Y \subseteq S$.
The inclusion $\VA [\Gamma_{X \cap Y}] \subseteq \VA [\Gamma_X] \cap \VA [\Gamma_Y]$ is obvious, hence we only need to prove the reverse inclusion.
Let $g \in \VA [\Gamma_X] \cap \VA [\Gamma_Y]$.
Let $w =\pi_K (g)$.
We have $w \in W [\Gamma_X] \cap W [\Gamma_Y] = W [\Gamma_{X \cap Y}]$, hence $\iota_W (w) \in \VA [\Gamma_{X \cap Y}]$.
Let $h = g\, \iota_W (w)^{-1}$.
Note that $h \in \VA [\Gamma_X] \cap \KVA [\Gamma] = \KVA [\Gamma_X]$ and $h \in \VA [\Gamma_Y] \cap \KVA [\Gamma] = \KVA [\Gamma_Y]$.
By Theorem \ref{thm2_15}, Lemma \ref{lem3_8}, and Proposition \ref{prop2_12}, we have:
\[
h \in \KVA [\Gamma_X] \cap \KVA [\Gamma_Y] = A [\hat \Gamma_{\Phi[\Gamma_X]}] \cap A [\hat \Gamma_{\Phi[\Gamma_Y]}] = A [\hat \Gamma_{ \Phi[\Gamma_X]\cap \Phi[\Gamma_Y]}]\,.
\]
From Lemma \ref{lem3_5}, we have:
\[ 
\Phi[\Gamma_X] \cap \Phi[\Gamma_Y] = \Phi [\Gamma] \cap V_X \cap V_Y = \Phi [\Gamma] \cap V_{X \cap Y} = \Phi [\Gamma_{X \cap Y}]\,.
\]
Thus, $h \in A [\hat \Gamma_{\Phi[\Gamma_{X \cap Y}]}] \subseteq \VA [\Gamma_{X \cap Y}]$, hence $g = h \, \iota_W (w) \in \VA [\Gamma_{X \cap Y}]$.
\end{proof}

As promised in the introduction, we conclude this section with a simple example of two parabolic subgroups of $\VA [\Gamma]$ whose intersection is not a parabolic subgroup.

\begin{expl}\label{expl3_9}
Let $\Gamma$ be the Coxeter graph $A_2$ illustrated in Figure \ref{fig3_1}.
We denote its vertex set by $S = \{s,t\}$. 
Let $X = \{s\}$, $Y = \{t\}$, $g = \sigma_s \sigma_t$, $P = \VA [\Gamma_Y]$, and $Q = g \VA [\Gamma_X] g^{-1}$.
We have $\pi_K (\sigma_t) = 1$, $\pi_K (\tau_t) = t$, $\pi_K (g \sigma_s g^{-1}) = 1$, and $\pi_K (g \tau_s g^{-1}) = s$, hence $\pi_K (P \cap Q) \subseteq W[\Gamma_Y] \cap W[\Gamma_X] = \{1\}$.
Note that, if the image of a parabolic subgroup of $\VA [\Gamma]$ under the homomorphism $\pi_K$ is $\{ 1 \}$, then this parabolic subgroup must be trivial.
In particular, if $P \cap Q$ were a parabolic subgroup, then we would have $P \cap Q = \{1\}$.
However, $\sigma_t = g \sigma_s g^{-1} \in P \cap Q$ and $\sigma_t \neq 1$, hence $P \cap Q$ is not a parabolic subgroup.
\end{expl}

\begin{figure}[ht!]
\begin{center}
\includegraphics[width=1.5cm]{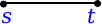}
\caption{Coxeter graph $A_2$}\label{fig3_1}
\end{center}
\end{figure}


\section{Word problem}\label{sec4}

We denote by $\AA = \SS \sqcup \SS^{-1} \sqcup \TT$ the standard generating set for $\VA [\Gamma]$, viewed as a monoid (note that $\TT^{-1} = \TT$).
The following two lemmas are preliminary results for the proof of Theorem \ref{thm1_6}.
For an element $g$ in $\KVA [\Gamma]$, they allow to move from an expression of $g$ over $\AA$ to an expression of $g$ over $\{ \delta_\beta \mid \beta \in \Phi [\Gamma]\}$, and vice versa.
Since the set $\Phi [\Gamma]$ is generally infinite, these two transitions are not immediate.
Lemma \ref{lem4_1} is proven in \cite[Proposition 5.2]{BPT23}.

\begin{lem}[Bellingeri--Paris--Thiel \cite{BPT23}]\label{lem4_1}
There exists an algorithm that, given a word $\omega$ over $\AA$ representing an element $g \in \KVA [\Gamma]$, determines:
\begin{itemize}
\item[(i)]
a finite subset $\XX \subseteq \Phi [\Gamma]$ such that $g \in A [\hat \Gamma_\XX]$;
\item[(ii)]
the Coxeter matrix $\hat M_\XX$ of $\hat \Gamma_\XX$;
\item[(iii)]
a word $\mu$ over $\{ \delta_\beta, \delta_\beta^{-1} \mid \beta \in \XX\}$ representing $g$.
\end{itemize}
\end{lem}

\begin{lem}\label{lem4_2}
There exists an algorithm that, given $\beta \in \Phi [\Gamma]$, determines a word $\eta$ over $S$ and an element $s \in S$ such that $\beta = \overline{\eta} (\alpha_s)$, where $\overline{\eta}$ denotes the element of $W [\Gamma]$ represented by $\eta$.
\end{lem}

\begin{proof}
Let $\beta \in \Phi [\Gamma]$.
First, assume that $\beta \in \Phi^+ [\Gamma]$.
We determine a word $\eta \in S^*$ and an element $s \in S$ such that $\beta = \overline{\eta} (\alpha_s)$ by induction on $\dpt (\beta)$.
If $\dpt (\beta) = 1$, then there exists $s \in S$ such that $\beta = \alpha_s$.
In this case, it suffices to set $\eta = 1$.
Suppose that $\dpt (\beta) \ge 2$, and that the induction hypothesis holds. 
By testing each element of $S$, we can find $t \in S$ such that $\langle \beta, \alpha_t \rangle > 0$.
Recall that $S$ is finite, hence this operation can be performed in finitely many steps.
By Lemma \ref{lem3_2}, it follows that $\dpt(t(\beta)) = \dpt(\beta)-1$.
By the induction hypothesis, we can determine a word $\eta' \in S^*$ and an element $s \in S$ such that $t(\beta) = \overline{\eta'}(\alpha_s)$.
Then $\beta = \overline{\eta}(\alpha_s)$, where $\eta = t \eta'$.

Now, suppose that $\beta \in \Phi^- [\Gamma]$.
From the argument above, we can find a word $\eta' \in S^*$ and an element $s \in S$ such that $- \beta = \overline{\eta'} (\alpha_s)$.
Then $\beta = \overline{\eta} (\alpha_s)$, where $\eta = \eta' s$.
\end{proof}

\begin{proof}[Proof of Theorem \ref{thm1_6}]
We assume that the word problem in $\VA [\Gamma]$ is solvable.
Let $X \subseteq S$.
Let $\SS_X = \{ \sigma_x \mid x \in X\}$, $\TT_X = \{ \tau_x \mid x \in X \}$, and $\AA_X = \SS_X \sqcup \SS_X^{-1} \sqcup \TT_X$.
Let $\omega_0 \in \AA^*$.
Or goal is to decide whether $\overline{\omega_0} \in \VA [\Gamma_X]$, and if so, to determine a word $\mu_0 \in \AA_X^*$ such that $\overline{\omega_0} = \overline{\mu_0}$.
We set $g_0 = \overline{\omega_0} \in \VA [\Gamma]$.

Let $\pi_K^* : \AA^* \to S^*$ be the monoid homomorphism that maps $\sigma_s$ and $\sigma_s^{-1}$ to $1$ and $\tau_s$ to $s$ for all $s \in S$.
It is clear that $\overline{\pi_K^* (\omega_0)} = \pi_K (g_0)$ in $W [\Gamma]$.
By Proposition \ref{prop2_5}, we can decide whether $\overline{\pi_K^* (\omega_0)} = \pi_K (g_0) \in W [\Gamma_X]$, and if so, determine a word $\eta_0 \in X^*$ such that $\overline{\eta_0} = \pi_K (g_0)$.
If $\pi_K (g_0) \not \in W [\Gamma_X]$, then $g_0 \not \in \VA [\Gamma_X]$, and the algorithm stops here.

So, we can assume that $\pi_K (g_0) \in W [\Gamma_X]$ and that we have obtained a word $\eta_0 \in X^*$ such that $\overline{\eta_0} = \pi_K (g_0)$.
Let $\iota_W^* : S^* \to \AA^*$ be the monoid homomorphism defined by $\iota_W^* (s) = \tau_s$ for all $s \in S$.
Let $\omega_1 = \omega_0 \, \iota_W^*(\eta_0)^{-1}$ and let $g_1 = \overline{\omega_1}$. 
Note that $\iota_W^* (\eta_0) \in \AA_X^*$ and $\pi_K (g_1) = 1$, that is, $g_1 \in \KVA [\Gamma]$.
Note also that $g_0 \in \VA [\Gamma_X]$ if and only if $g_1 \in \KVA [\Gamma] \cap \VA [\Gamma_X] = \KVA[\Gamma_X]$.

Applying the algorithm from Lemma \ref{lem4_1} to the word $\omega_1$, we obtain:
\begin{itemize}
\item[(i)]
a finite subset $\XX \subseteq \Phi [\Gamma]$ such that $g_1 \in A [\hat \Gamma_\XX]$;
\item[(ii)]
the Coxeter matrix $\hat M_\XX$ of $\hat \Gamma_\XX$;
\item[(iii)]
a word $\nu_1$ over $\{ \delta_\beta, \delta_\beta^{-1} \mid \beta \in \XX\}$ representing $g_1$.
\end{itemize}
Furthermore, using the algorithm from Corollary \ref{corl3_6}, we can determine $\YY = \XX \cap \Phi [\Gamma_X]$.
Note that, by Proposition \ref{prop2_12}, we have $A [\hat \Gamma_\XX] \cap A [\hat \Gamma_{\Phi[\Gamma_X]}] = A [\hat \Gamma_\YY]$.
In particular, we have $g_1 \in \KVA [\Gamma_X] = A [\hat \Gamma_{\Phi [\Gamma_X]}]$ if and only if $g_1 \in A [\hat \Gamma_\YY]$.

The word problem in $A[\hat \Gamma_\XX$] can be solved as follows.
For each $\beta \in \XX$, we apply the algorithm of Lemma \ref{lem4_2} to determine a word $\kappa \in S^*$ and an element $s \in S$ such that $\beta = \overline{\kappa} (\alpha_s)$, and we set $\xi (\beta) = \iota_W^* (\kappa) \sigma_s \iota_W^* (\kappa)^{-1} \in \AA^*$.
Let $\xi^* : \{ \delta_\beta, \delta_\beta^{-1} \mid \beta \in \XX \}^* \to \AA^*$ be the monoid homomorphism that maps $\delta_\beta^{\pm 1}$ to $\xi (\beta)^{\pm 1}$.
Let $\nu$ be a word over $\{ \delta_\beta, \delta_\beta^{-1} \mid \beta \in \XX \}$.
By construction, we have $\overline{\nu} = \overline{\xi^* (\nu)}$.
So, using the word problem in $\VA [\Gamma]$, we can decide whether $\overline{\nu} =1$ or not.

By Proposition \ref{prop2_13}, we can decide whether $g_1 \in A [\hat \Gamma_\YY]$, and if so, determine a word $\nu_1' \in \{ \delta_\beta, \delta_\beta^{-1} \mid \beta \in \YY \}^*$ representing $g_1$.
If $g_1 \not \in A [\hat \Gamma_\YY]$, then $g_0 \not \in \VA [\Gamma_X]$, and the algorithm stops here.
If $g_1 \in A [\hat \Gamma_\YY]$, then $g_0 \in \VA [\Gamma_X]$, and we still need to determine a word $\mu_0 \in \AA_X^*$ representing $g_0$.

So, we can assume that $g_1 \in A [\hat \Gamma_\YY]$ and that we have obtained a word $\nu_1' \in \{ \delta_\beta, \delta_\beta^{-1} \mid \beta \in \YY \}^*$ representing $g_1$.
We define the map $\xi_X^* : \{ \delta_\beta, \delta_\beta^{-1} \mid \beta \in \YY \}^* \to \AA_X^*$ by restricting the previous construction of $\xi$ to $\YY$ and $\AA_X^*$ as follows.
For each $\beta \in \YY \subseteq \Phi[\Gamma_X]$, we apply Lemma \ref{lem4_2} to determine a word $\kappa \in X^*$ and an element $x \in X$ such that $\beta = \overline{\kappa} (\alpha_x)$, and we set $\xi_X (\beta) = \iota_W^* (\kappa) \sigma_x \iota_W^* (\kappa)^{-1} \in \AA_X^*$.
Let $\xi_X^* : \{ \delta_\beta, \delta_\beta^{-1} \mid \beta \in \YY \}^* \to \AA_X^*$ be the monoid homomorphism that maps $\delta_\beta^{\pm 1}$ to $\xi_X (\beta)^{\pm 1}$ for each $\beta \in \YY$.
By construction, for any word $\nu \in \{ \delta_\beta, \delta_\beta^{-1} \mid \beta \in \YY \}^*$, we have $\overline{\nu} = \overline{\xi_X^* (\nu)}$.
In particular, $\xi_X^* (\nu_1')$ is a word over $\AA_X$ representing $g_1 = \overline{\omega_1}$.
Now, set $\mu_0 = \xi_X^* (\nu_1')\,\iota_W^* (\eta_0)$.
Then $\mu_0$ is a word over $\AA_X$ representing $g_0$.
\end{proof}

\begin{proof}[Proof of Theorem \ref{thm1_4}]
We assume that the word problem in $\VA[\Gamma_X]$ is solvable for each free of infinity subset $X \subseteq S$, and we show that the word problem in $\VA[\Gamma]$ is solvable.
We argue by induction on the number $e_\infty(\Gamma)$ of edges of $\Gamma$ labeled with $\infty$.
If $e_\infty(\Gamma) = 0$, then $S$ itself is free of infinity, hence the word problem in $\VA[\Gamma]$ is solvable by assumption.
Now, assume that $e_\infty(\Gamma) \ge 1$, and that the induction hypothesis holds.
Choose $s,t \in S$ such that $s \neq t$ and $m_{s,t} = \infty$.
Let $X = S \setminus \{s\}$, $Y = S \setminus \{t\}$, and $Z = S \setminus \{s,t\}$.
From the presentation of $\VA [\Gamma]$, it follows that:
\[
\VA [\Gamma] = \VA[\Gamma_X] *_{\VA [\Gamma_Z]} \VA [\Gamma_Y]\,.
\]
By the induction hypothesis, the word problem is solvable in both $\VA [\Gamma_X]$ and $\VA [\Gamma_Y]$. 
By Theorem \ref{thm1_6}, it follows that the strong membership problem for $\VA[\Gamma_Z]$ in $\VA[\Gamma_X]$ and the strong membership problem for $\VA[\Gamma_Z]$ in $\VA[\Gamma_Y]$ are solvable.
By Proposition \ref{prop2_18}, we conclude that the word problem in $\VA[\Gamma]$ is solvable.
\end{proof}


\end{document}